\documentclass[reqno]{amsart}
\usepackage{graphicx} 
\usepackage[utf8]{inputenc}
\usepackage{amsmath}

\usepackage{tikz}

\usetikzlibrary{shapes.geometric}
\usepackage{amsthm}

\usepackage{amsfonts}
\usepackage{caption}
\usepackage{amssymb}
\usepackage{verbatim}
\usepackage{mathtools}
\usepackage{subcaption}
\usepackage{changepage}
\newtheorem{theorem}{Theorem}
\newtheorem{lemma}[theorem]{Lemma}

\newtheorem{corollary}[theorem]{Corollary}
\newtheorem{definition}[theorem]{Definition}
\numberwithin{equation}{section}

\title{Values of Ducci Periods for Sequences on $\mathbb{Z}_m^n$}

\author{Mark L. Lewis }
\address{Department of Mathematical Sciences\\
Kent State University\\
Kent, OH 44242}
\email{lewis@math.kent.edu}
\author{Shannon M. Tefft}
\address{Department of Mathematical Sciences\\
Kent State University\\
Kent, OH 44242}
\email{stefft@kent.edu}


\date{February 2025}
\begin{document}
\begin{abstract}
Let $D: \mathbb{Z}_m^n \to \mathbb{Z}_m^n$ be defined so that 
\[D(x_1, x_2, ..., x_n)=(x_1+x_2 \; \text{mod} \; m, x_2+x_3 \; \text{mod} \; m, ..., x_n+x_1 \; \text{mod} \; m).\]
We call $D$ the Ducci function and the sequence $\{D^{\alpha}(\mathbf{u})\}_{\alpha=0}^{\infty}$ the Ducci sequence of $\mathbf{u}$ for $\mathbf{u} \in \mathbb{Z}_m^n$. Every Ducci sequence enters a cycle, so we can let $\text{Per}(\mathbf{u})$ be the number of tuples in the Ducci cycle of $\mathbf{u}$, or the period of $\mathbf{u}$. In this paper, we will look at what different possible values of $\text{Per}(\mathbf{u})$ we can have and some conditions that if $\mathbf{u}$ meets at least one of them, $\mathbf{u}$ will generate a period smaller than the maximum period. 
\end{abstract}
\maketitle
\section{Introduction}\label{intro}
\indent We begin by introducing an endomorphism on $\mathbb{Z}_m^n$ known as the Ducci function, $D$, with \cite{Breuer1, Ehrlich, Glaser} being a few examples that call it this. Namely, we define $D$ so that
\[D(x_1, x_2, ..., x_n)=(x_1+x_2 \; \text{mod} \; m, x_2+x_3 \; \text{mod} \; m, ..., x_n+x_1 \; \text{mod} \; m).\]
We also call the sequence $\{D^{\alpha}(\mathbf{u})\}_{\alpha=0}^{\infty}$ the \textbf{Ducci sequence of} $\mathbf{u}$ for $\mathbf{u} \in \mathbb{Z}_m^n$. 

\indent As \cite{Breuer1} notes, every Ducci sequence on $\mathbb{Z}_m^n$ enters a cycle because $\mathbb{Z}_m^n$ is finite. More specifically, we define
\begin{definition}
The \textbf{Ducci cycle of} $\mathbf{u}$ is  
\[\{\mathbf{v} \mid \exists \alpha \in \mathbb{Z}^+ \cup \{0\}, \beta \in \mathbb{Z}^+  \ni \mathbf{v}=D^{\alpha+\beta}(\mathbf{u})=D^{\alpha}(\mathbf{u})\}\].
The \textbf{length of} $\mathbf{u}$, $\mathbf{Len(u)}$, is the smallest $\alpha$ satisfying the equation 
\[\mathbf{v}=D^{\alpha+\beta}(\mathbf{u})=D^{\alpha}(\mathbf{u})\]
 for some $v \in \mathbb{Z}_m^n$ \and the \textbf{period of} $\mathbf{u}$, $\mathbf{Per(u)}$, is the smallest $\beta$ that satisfies the equation. 
\end{definition}

\indent To demonstrate this definition, we look at the Ducci sequence of $(0,0,2) \in \mathbb{Z}_4^3$: $(0,0,2), (0,2,2), (2,0,2), (2,2,0), (0,2,2)$. Here we see that $\text{Len}(0,0,2)=1$ and $\text{Per}(0,0,2)=3$. 

\indent A particular Ducci sequence of interest is the Ducci sequence of $(0,0,...,0,1)$ in $\mathbb{Z}_m^n$, known as the \textbf{basic Ducci sequence}. The basic Ducci sequence is first introduced by \cite{Ehrlich} on page 302 and is also used by  \cite{Breuer1, Dular,Glaser}. 
Define 
\[\mathbf{P_m(n)}=\text{Per}(0,0,...,0,1)\]
 and
  \[\mathbf{L_m(n)}=\text{Len}(0,0,...,0,1).\] This notation is similar to the notation that \cite{Breuer1} uses in Definition 5, and $P_m(n)$ is also similar to the notation used by \cite{Ehrlich, Glaser}. 
This sequence is important is because of Lemma 1 in \cite{Breuer1}, which says that for $\mathbf{u} \in \mathbb{Z}_m^n$, $\text{Len}(\mathbf{u}) \leq L_m(n)$ and $\text{Per}(\mathbf{u})|P_m(n)$. This provides maximum values for the length and periods of tuples in $\mathbb{Z}_m^n$. The notation $P_m(n)$ for the maximum period on $\mathbb{Z}_m^n$ is also used by \cite{Breuer2} on page 858.

\indent Discussing a value for a maximum period also begs the question about when tuples do or do not generate this period. In this paper, we would like to focus on some different cases where this happens. We will begin by examining a smaller period that appears when $n=3$ by proving the following Lemma:

\begin{lemma} \label{TuplesLength6}
Suppose $\mathbf{u}=(x_1, x_2, x_3) \in \mathbb{Z}_m^3$ such that $x_1+x_2+x_3\equiv 0 \; \text{mod} \; m$. Then 
\begin{enumerate}
    \item Every tuple in the Ducci sequence of $\mathbf{u}, \; \mathbf{v}=(x_1',x_2',x_3')$, satisfies 
    \[x_1'+x_2'+x_3'\equiv 0 \; \text{mod} \;m.\]
    \item $D^2(\mathbf{u})=H(\mathbf{u})$.
    \item $\mathbf{u}$ belongs to a Ducci cycle.
    \item If $m$ odd, then unless $x_1=x_2=x_3$,  $Per(\mathbf{u})=6$.
    \item If $m$ is even and the $x_i$ are not all equal, then either $\text{Per}(\mathbf{u})=6$, or $\text{Per}(\mathbf{u})=3$ and $x_1, x_2, x_3 \in \{0, \frac{m}{2}\}$.
\end{enumerate}
\end{lemma}
where $H$ is an endomorphism on $\mathbb{Z}_m^n$ defined so 
\[H(x_1, x_2, ..., x_n)=(x_2, x_3, ..., x_n, x_1).\]
The function $H$ is first defined by \cite{Ehrlich} on page 302, and is also used by \cite{Breuer1, Glaser}.

\indent After this, we would like to narrow our focus to when $m$ is prime to show that $\text{Per}(\mathbf{u})=P_m(n)$ for most $\mathbf{u} \in \mathbb{Z}_m^3$:
\begin{theorem}\label{BigThm_prime}
Let $n=3$, $m$ an odd prime, and $\mathbf{u}=(x_1, x_2, x_3) \in \mathbb{Z}_m^3$. If 
\[x_1+x_2+x_3 \not \equiv 0 \; \text{mod} \; m\]
and $x_1, x_2, x_3$ not all equal, then $Per(\mathbf{u})=P_m(3)$.
\end{theorem}
\indent We will then take a look at some specific cases where $n,m$ are both prime but $n$ may not be $3$ and what periods are possible for these cases, as well as how we were able to determine those periods. 
Finally, we would like to look at a specific case where $n=m$ is prime, and determine all of the possible periods for a tuple $\mathbf{u} \in \mathbb{Z}_m^n$ for this case:
\begin{theorem}\label{BigThm_n_not_3}
    Let $n=m=p$ be an odd prime and let $\delta$ be the order of $2$ mod $p$. Then there are exactly 3 distinct period lengths for $\mathbf{u} \in \mathbb{Z}_p^p$:
    \begin{enumerate}
        \item $Per(\mathbf{u})=1$, which can only happen when $\mathbf{u}=(0,0,...,0)$.\\
        \item $Per(\mathbf{u})=\delta$, which can only happen if $\mathbf{u}=(x,x,...,x)$ for some $x \in \mathbb{Z}_p$.\\
        \item $Per(\mathbf{u})=P_p(p)=p*\delta$.\\
        \end{enumerate}
\end{theorem}

\indent The work in this paper was done while the second author was a Ph.D. student at Kent State University under the advisement of the first author and will appear as part of the second author's dissertation. 

\section{Background}\label{Background}

The Ducci function on $\mathbb{Z}_m^n$ is first examined by \cite{Wong}, as well as \cite{Breuer1, Breuer2, Dular}. However, this endomorphism is based on another function, $\bar{D}$, where
 \[\bar{D}(x_1, x_2,..., x_n)=(|x_1-x_2|, |x_2-x_3|, ..., |x_n-x_1|)\]
 and either $\bar{D}: \mathbb{Z}^n \to \mathbb{Z}^n $ or $\bar{D}: (\mathbb{Z}^+ \cup \{0\})^n \to (\mathbb{Z}^+ \cup \{0\})^n$.
 A few sources that define Ducci in this way are \cite{Ehrlich, Freedman, Glaser, Furno}. Other papers, like \cite{Brown, Chamberland, Schinzel}, look at Ducci defined on $\mathbb{R}^n$, where we once more use the formula given for $\bar{D}$. 
 
 \indent For the sake of simplicity, we will refer to both Ducci on $\bar{D}:\mathbb{Z}^n \to \mathbb{Z}^n$ and $\bar{D}: (\mathbb{Z}^+ \cup \{0\})^n \to (\mathbb{Z}^+ \cup \{0\})^n$ as Ducci on $\mathbb{Z}^n$. This works because if $\mathbf{u} \in \mathbb{Z}^n$, then $\bar{D}(\mathbf{u}) \in (\mathbb{Z}^+ \cup \{0\})^n$. 
 
 \indent There are a few findings from Ducci on $\mathbb{Z}^n$ that we would like to note before moving on. First, every Ducci sequence on $\mathbb{Z}^n$ enters a cycle. This is discussed in \cite{Burmester, Ehrlich, Glaser, Furno}. Additionally, we have that all of the tuples belonging to a Ducci sequence satisfy the condition that all of their entries belong to $\{0,c\}$ for some $c \in \mathbb{Z}$, which is proved in Lemma 3 of \cite{Furno}. 
 
 
 \indent For the Ducci case on $\mathbb{Z}^n$, this makes the Ducci case on $\mathbb{Z}_2^n$ important because $D(\lambda \mathbf{u})=\lambda D(\mathbf{u})$, where we use the same definition as Ducci on $\mathbb{Z}_m^n$.
 
 \indent We now return to Ducci on $\mathbb{Z}_m^n$ and our original definition for $D$. We begin by presenting a few more definitions that we would like to use throughout the paper. If $\mathbf{u}, \mathbf{v} \in \mathbb{Z}_m^n$, then $\mathbf{v}$ is called a \textbf{predecessor} of $\mathbf{u}$ if $D(\mathbf{v})=\mathbf{u}$.
 
 \indent Let $K(\mathbb{Z}_m^n)$ be the set of all tuples in $\mathbb{Z}_m^n$ that belong to a Ducci cycle for some $\mathbf{u} \in \mathbb{Z}_m^n$. This definition is first introduced in Definition 4 of \cite{Breuer1}. This Definition also states that $K(\mathbb{Z}_m^n)$ is a subgroup of $\mathbb{Z}_m^n$, and a proof is in Theorem 1 of \cite{Paper1}. 
 
 \indent There are a couple of papers that have looked at the value of $P_m(n)$. 
 One of these papers is \cite{Breuer1}, which in Proposition 4 shows that if $p$ is prime, then $P_p(p^kn_1)=p^kP_p(n_1)$.
 Theorem 8.2 of \cite{Breuer2}, shows that $P_m(n)$ can be broken down in terms of $P_p(n)$ where $m=p^l$ for some prime $p$ and $l \geq 1$. As an example of one of these break downs, Corollary 8.3 isolates to the case where $p \nmid 2n$, and $p^2 \nmid (2^{p-1}-1)$ to say $P_{p^l}(n)=p^{l-1}P_p(n)$. Theorem 3.2 of \cite{Dular} provides more significance to this result, by proving that if $m=m_1m_2$ where $gcf(m_1, m_2)=1$, then $P_m(n)=lcm\{P_{m_1}(n), P_{m_2}(n)\}$, which allows us to narrow our examination of the value of $P_m(n)$ to when $m$ is a prime power, and therefore to the formula that \cite{Breuer2} provides in Theorem 8.2. 
 
 \indent However, in this paper, we are not as interested in the actual value of $P_m(n)$. Rather, we would like to look at 
 the set $\{\text{Per}(\mathbf{u}) \; | \; \mathbf{u} \in \mathbb{Z}_m^n\}$ as a whole. How large is this set? For each value of $\text{Per}(\mathbf{u})$ in this set, how many different tuples generate that period? Can we find conditions $\mathbf{u}$ must meet to generate that period?

 \indent The one value of $\text{Per}(\mathbf{u})$ that is always a period for some tuple in $\mathbb{Z}_m^n$ for every $n,m$ is $\text{Per}(0,0,...,0)=1$. We define $\vec{\mathbf{0}}=(0,0,...,0)$. From Remark 2 of \cite{Breuer1}, we know that the only time $\text{Per}(\mathbf{u})=1$ is when the Ducci cycle of $\mathbf{u}$ is $\{\vec{\mathbf{0}}\}$. If the Ducci cycle of $\mathbf{u}$ is $\{\vec{\mathbf{0}}\}$, then we say that the Ducci sequence of $\mathbf{u}$ \textbf{vanishes}, a term that is first used on page 117 by \cite{Burmester} and also by \cite{Breuer1, Freedman}. 
 
 \indent One of the most well-known facts about Ducci on $\mathbb{Z}^n$ is that all Ducci sequences vanish if $n$ is a power of $2$. The first paper to prove this is \cite{Ciamberlini}, as noted by \cite{Brown1, Chamberland}. The papers \cite{Freedman, Furno} also credit \cite{Ciamberlini} for a proof of this. Unfortunately, we are unable to find a copy of \cite{Ciamberlini} to confirm this ourselves. A review of the paper can be found at \cite{CiamberliniReview}. Nonetheless, many other papers have proofs of this fact as well, including \cite{Ehrlich, Freedman, Miller, Pompili}. 
 
 \indent This can be extended to Ducci on $\mathbb{Z}_m^n$. If $n=2^k$ and $m=2^l$, then \cite{Wong} proves in (I) on page 103 that all tuples in $\mathbb{Z}_m^n$ vanish. This is proved again by \cite{Dular}. In Theorem 2 of \cite{Paper1}, we prove this once more, plus that $L_m(n)=2^{k-1}(l+1)$. Therefore, in this case, $\{\text{Per}(\mathbf{u}) \; | \; \mathbf{u} \in \mathbb{Z}_m^n\}=\{1\}$.
 
 \indent For another possible period, we look at the tuples where all of the entries are the same, nonzero value. We say a tuple satisfies the uniformity condition if all of its entries are nonzero and equal. We now prove a lemma about tuples that meet this condition:
 \begin{lemma}\label{uniformcondlemma}
 Suppose $\mathbf{u}=(x,x, ..., x)$ for some $ 0 < x <m$. Assume $m=2^lm_1>1$ where $m_1 \geq 1$ is odd and $l \geq 0$. 
\begin{enumerate} 
\item Assume $m_1=1$. Then $\mathbf{u}$ vanishes for every $x \in \mathbb{Z}_m$.\\
 \item Assume $m_1>1$. Let $\delta$ be the multiplicative order of $2 \; \text{mod} \; m_1$   then $\text{Per}(\mathbf{u})=\delta$ when $m_1 \nmid x$ and $\mathbf{u}$ vanishes when $m_1|x$.
 \end{enumerate}
 \end{lemma} 
 \begin{proof}
 We first note that $D^{\alpha}(x,x,..., x)=2^{\alpha}(x, x,..., x)$.
 
 \indent $\mathbf{(1)}$: Assume $m=2^l$. Then $D^l(x,x, ..., x)=2^l(x,x, ..., x)$, which is $\vec{\mathbf{0}}$. Therefore, $(x, x, ..., x)$ vanishes for all $x \in \mathbb{Z}_m$.
 
 \indent $\mathbf{(2)}$: First consider when $l=0$ so $m=m_1>1$ is odd. Let $\delta$ be the multiplicative order of $2 \; \text{mod} \; m$. Then 
 \[D^{\delta}(x,x, ..., x)=2^{\delta}(x,x, ..., x),\]
 or $(x,x, ..., x)$. Since $\delta$ is the smallest value where this can happen, $\text{Per}(\mathbf{u})= \delta$.
 
\indent Now let $m=2^lm_1$ where $m_1>1$ and $l>0$. Let $\delta$ be the multiplicative order of $2 \; \text{mod} \; m_1$. If $m_1|x$, then 
\[D^l(x,x, ...,x)=2^l(x,x,..., x),\]
which is $\vec{\mathbf{0}}$, so $\mathbf{u}$ vanishes.

\indent Now assume $m_1 \nmid x$ and $2^l|x$. Then 
\[D^{\delta}(x,x,...,x)=2^{\delta}(x,x,...,x).\]
This is $(x,x, ..., x)$. Once more, $\delta$ is the smallest value where this happens, so $\text{Per}(\mathbf{u})=\delta$. 

\indent Suppose now that $x$ is odd or there exists $1 \leq l_1 <l$ such that $x \equiv 0 \; \text{mod} \; 2^{l_1}$ and $x \not \equiv 0 \; \text{mod} \; 2^{l_1+1}$. Then 
\[D^{l-l_1}(x,x,..., x)=2^{l-l_1}(x,x,...,x).\]
Now $2^{l-l_1}x \equiv 0 \; \text{mod} \; 2^l$. Therefore, 
\[\text{Per}(2^{l-l_1}x, 2^{l-l_1}x, ..., 2^{l-l_1}x)= \delta.\]
Notice
\[\text{Per}(x,x,..., x)=\text{Per}(2^{l-l_1}x, 2^{l-l_1}x, ..., 2^{l-l_1}x)\]
because $2^{l-1}(x,x,...,x)$ is in the Ducci sequence of $(x,x,...,x)$.
We conclude that $\text{Per}(x,x, ..., x)=\delta$.
 \end{proof}
 
 \indent This therefore opens up one other possible period when $m$ is not a power of $2$. Note that this may not always be a distinct period from $P_m(n)$. For example, $P_{11}(5)=10$ and the multiplicative order of $2 \; \text{mod} \; 11$ is also $10$. Since $11$ is prime, all tuples that satisfy the uniformity condition in $\mathbb{Z}_{11}^5$ still generate the maximum period $P_{11}(5)$.
 
\indent To discuss the next type of possible period that can arise, we return to the example we gave in Section \ref{intro}. To do so, we build a transition graph that maps the Ducci sequences of all tuples in $\mathbb{Z}_4^3$, and then look at the connected component containing $(0,0,2)$, given in Figure \ref{Transgraph}.

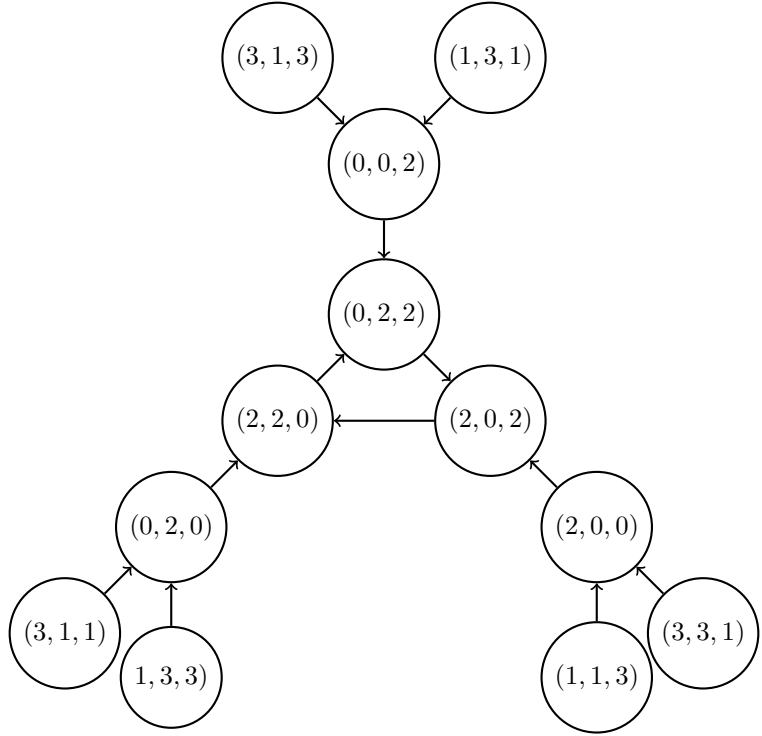
\begin{figure}
\centering

\begin{tikzpicture}[node distance={20mm}, thick, main/.style = {draw, circle}]

\node[main](1){$(0,2,2)$};
\node[main](2)[below right of=1]{$(2,0,2)$};
\node[main](3)[below left of=1]{$(2,2,0)$};

\node[main](4)[above of=1]{$(0,0,2)$};
\node[main](5)[above left of=4]{$(3,1,3)$};
\node[main](6)[above right of=4]{$(1,3,1)$};

\node[main](7)[below right of=2]{$(2,0,0)$};
\node[main](8)[below of=7]{$(1,1,3)$};
\node[main](9)[below right of=7]{$(3,3,1)$};

\node[main](10)[below left of=3]{$(0,2,0)$};
\node[main](11)[below left of=10]{$(3,1,1)$};
\node[main](12)[below of=10]{$1,3,3)$};

\draw[->](1)--(2);
\draw[->](2)--(3);
\draw[->](3)--(1);

\draw[->](4)--(1);
\draw[->](5)--(4);
\draw[->](6)--(4);

\draw[->](7)--(2);
\draw[->](8)--(7);
\draw[->](9)--(7);

\draw[->](10)--(3);
\draw[->](11)--(10);
\draw[->](12)--(10);

\end{tikzpicture}
\caption{Transition Graph for $\mathbb{Z}_4^3$}\label{Transgraph}
\end{figure}

\indent We note first that $P_4(3)=6$. From Figure \ref{Transgraph}, we can see that $\text{Per}(0,0,2)=3$, giving us another example of a tuple that generates a smaller period than the maximum. It is worth noting that the tuples in this cycle belong to the subgroup $\mathbb{Z}_2^3$, as well as $(0,0,2), (0,2,0), (2,0,0)$. In fact, the portion of Figure \ref{Transgraph} that these tuples make up in $\mathbb{Z}_4^3$ is  the same as the connected component of the transition graph of $\mathbb{Z}_2^3$ that contains the basic Ducci sequence, except that in $\mathbb{Z}_2^3$, the entries where $2$ appears are replaced with $1$. More generally speaking, we have

\begin{lemma}\label{periods_divide}
Assume $m_1|m$. Then for every $\mathbf{u} \in \mathbb{Z}_{m_1}^n$, there exists $\mathbf{v} \in \mathbb{Z}_m^n$ such that $\text{Per}(\mathbf{v})=\text{Per}(\mathbf{u})$.
\end{lemma}
It follows from Proposition 3.1 of \cite{Dular} that $P_{m_1}(n)|P_m(n)$ whenever $m_1|m$, but this lemma can extend this finding.
\begin{proof}[Proof of Lemma \ref{periods_divide}]
Assume $m_1|m$ and let $\mathbf{u}=(x_1, x_2, ..., x_n) \in \mathbb{Z}_{m_1}^n$. Without loss of generality, we may assume $\mathbf{u} \in K(\mathbb{Z}_m^n)$, otherwise we can take a tuple in the Ducci cycle of $\mathbf{u}$ instead. 

\indent Let $d=\text{Per}(\mathbf{u})$. Then $D^d(\mathbf{u})=\mathbf{u}$ in $\mathbb{Z}_{m_1}^n$, so isolating the $s$th entry, we have
\[a_{d,1}x_s+a_{d,2}x_{s+1}+ \cdots + a_{d,n}x_{s-1} \equiv x_s \; \text{mod} \; m_1.\]
Write the left side of this equivalency as $x_s + \delta_sm_1$ for some  integer $\delta_s$. 

Let $\mathbf{v}=\displaystyle{\frac{m}{m_1}}(x_1, x_2, ..., x_n) \in \mathbb{Z}_m^n$ where we maintain that $0 \leq x_i <m_1$, $x_i \in \mathbb{Z}_m$. Then, the $s$th entry of $D^d(\mathbf{v})$ is 
\[\frac{m}{m_1}(a_{d,1}x_s+ a_{d,2}x_{s+1}+ \cdots a_{d,n}x_{s-1}).\]
This is 
\[\frac{m}{m_1}(x_s+\delta_sm_1),\]
which is equivalent to $\displaystyle{\frac{m}{m_1}}x_s \; \text{mod} \; m$. Since this is true for every $1 \leq s \leq n$, $D^d(\mathbf{v})=\mathbf{v}$ and $\text{Per}(\mathbf{v})|d$. If $\text{Per}(\mathbf{v})<d$, this would imply that $\text{Per}(\mathbf{u})< d$, which forces $\text{Per}(\mathbf{v})=\text{Per}(\mathbf{u})$. 
\end{proof}

\indent This gives us that if $m_1|m$, then 
\[\{\text{Per}(\mathbf{u}) \; | \; \mathbf{u} \in \mathbb{Z}_{m_1}^n\} \subseteq \{\text{Per}(\mathbf{v}) \; | \; \mathbf{v} \in \mathbb{Z}_m^n\},\]
and provides many other possible periods when $m$ is not prime. 

\indent Notice that in Figure \ref{Transgraph}, there are other tuples branching off of the tuples in $\mathbb{Z}_2^3$ that belong to $\mathbb{Z}_4^3$, but not $\mathbb{Z}_2^3$. In this case, this must happen because $L_4(3)>L_2(3)$. Similarly, this will happen any time $L_m(n)>L_{m_1}(n)$. In other cases when $n$ is even, other tuples will appear in the transition graph when moving from $\mathbb{Z}_{m_1}^n$ to $\mathbb{Z}_m^n$ because every tuple with a predecessor will go from having $m_1$ predecessors to $m$ predecessors. This is because of Theorem 4 of \cite{Paper1}, which says if $n$ is even, then every tuple in $\mathbb{Z}_m^n$ that has a predecessor has exactly $m$ predecessors. As a result, in both of these cases, these new tuples that appear also have the smaller period, despite not belonging to $\mathbb{Z}_{m_1}^n$.

\indent We have a similar result when $n$ is not prime. If $n=n_1n_2$, $(x_1, x_2, ..., x_{n_1}) \in \mathbb{Z}_m^{n_1}$, and $(x_1, x_2, ..., x_{n_1}, x_1, ..., x_{n_1}, ..., x_1, ..., x_{n_1}) \in \mathbb{Z}_m^n$ where $(x_1, x_2, ..., x_{n_1})$ is repeated $n_2$ times, then 
\[\text{Per}(x_1, x_2, ..., x_{n_1}, x_1, ..., x_{n_1}, ..., x_1, ..., x_{n_1})=\text{Per}(x_1, x_2, ..., x_{n_1}).\]

This is because
\[D(x_1, x_2, ..., x_{n_1}, x_1, ..., x_{n_1}, ..., x_1, ..., x_{n_1})\]\[=(D(x_1, x_2, ..., x_{n_1}), D(x_1, ..., x_{n_1}), ..., D(x_1, ..., x_{n_1})).\]

\indent There is one last possible type of period that we would like to discuss. 
Consider a tuple, $(x_1, x_2, ..., x_n)$. We say it satisfies the sum condition if
\[x_1+x_2+ \cdots + x_n \equiv 0 \; \text{mod} \; m.\]
 For some $n,m$, all tuples satisfying the sum condition generate a smaller period than $P_m(n)$. If $n=3$, tuples satisfying the sum conditions always have a period of $6$, as we wish to prove in Lemma \ref{TuplesLength6}. However, there are $n,m$ where this does not happen. For example, proving Theorem \ref{BigThm_n_not_3} will show this is not the case when $n=m$ is prime. We will examine this topic more for when $n \neq 3$ where both $n,m$ are prime in Section \ref{sec_not_n=3}.


\indent Before moving on, we would like to note that for certain $n,m$, there are other tuples $\mathbf{u} \in \mathbb{Z}_m^n$ such that $\text{Per}(\mathbf{u})<P_m(n)$ where $\mathbf{u}$ does not meet any of the conditions that we have discussed in this section.

\section{The Coefficients $a_{r,s}$}

\indent In order to further investigate Ducci sequences, we will introduce a tool that helps us visualize what a tuple in the Ducci sequence of $(x_1, x_2, ..., x_n)$ looks like. First, examine the first few terms of the Ducci sequence of $(x_1, x_2, ..., x_n)$:
\[(x_1, x_2, ..., x_n)\]
     \[(x_1+x_2, x_2+x_3, ..., x_n+x_1)\]
     \[(x_1+2x_2+x_3, x_2+2x_3+x_4, ..., x_n+2x_1+x_2\]
     \[(x_1+3x_2+3x_3+x_4, x_2+3x_3+3x_4+x_5, ..., x_n+3x_1+3x_2+x_3)\]
     \[\vdots\]
     
     One can see that there is a pattern to the coefficients on the $x_i$ in each entry of a tuple in the sequence. For example, in $D^3(x_1, x_2, ..., x_n)$, the coefficients on $x_1, x_2, x_3, x_4$ all occur in the other entries of $D^3(x_1, x_2, ..., x_n)$, though not necessarily on the same $x_i$ as before. We take advantage of this pattern and define $a_{r,s}$ to be the coefficient on $x_s$ in the first entry of $D^r(x_1, x_2, ..., x_n)$, for $r \geq 0$ and $1 \leq s \leq n$ where $s$ is reduced modulo $n$, allowing the $s$ coordinate to be $n$ when $s$ is a multiple of $n$. Additionally, $a_{r,s}$ is the coefficient on $x_{s-i+1}$ in the $i$th entry of $D^r(x_1, x_2, ..., x_n)$. A discussion of why we can make these assumptions about $a_{r,s}$ is provided on page 6 of \cite{Paper1}. 
     
 \indent Theorem 5 of \cite{Paper1} also provides a few facts about $a_{r,s}$ that we will use:
 \begin{itemize}
 \item For $r \geq 1$, $a_{r,s}=a_{r-1,s}+a_{r-1,s-1}$.\\
 \item For $0 \leq r <n$, $a_{r,s}=\displaystyle{\binom{r}{s-1}}$.\\
 \item For $r \geq 1$, $t \geq 1$,
 \[a_{r+t,s}=\sum_{i=1}^n a_{t,i}a_{r,s-i+1}.\]
 \end{itemize}
 
 \indent For the basic Ducci sequence in particular, we have:
 \[D^r(0,0,...,0,1)=(a_{r,n}, a_{r, n-1}, ..., a_{r,1}).\]
 
 \indent Because so much of this paper focuses on when $n=3$, we provide some theorems about the $a_{r,s}$ coefficients  specific to when $n=3$. In addition to this, because we will only be working with three different $a_{r,s}$ coefficients, we write them as:
 \[a_r=a_{r,1}\]
 \[b_r=a_{r,2}\]
 \[c_r=a_{r,3}.\]
 So for our basic Ducci sequence, $D^r(0,0,1)=(c_r, b_r, a_r)$. Rewriting our sum formula from Theorem 5 of \cite{Paper1}, we have 
 \begin{corollary} \label{coefficient_sum_n=3}\label{n=3_coeff_sum}
Let $n=3$. For $r> t  \geq 1$,
\[a_{r+t}=a_ta_{r}+b_tc_{r}+c_tb_{r}\]
\[b_{r+t}=a_tb_{r}+b_ta_{r}+c_tc_{r}\]
\[c_{r+t}=a_tc_{r}+b_tb_{r}+c_ta_{r}.\]
\end{corollary}  

\indent When $n=3$, $a_r, b_r,$ and $c_r$ are all very close in value, which we can clarify in the following lemma:
 \begin{lemma}\label{n=3_coeff_relat}
 Let $n=3$. Then
\begin{itemize}
    \item If $r \equiv 0 \; \text{mod} \; 6$, $a_r=b_r+1=c_r+1$.\\
    \item If $r \equiv 1 \; \text{mod} \; 6$, then $c_r=a_r-1=b_r-1$.\\
    \item If $r \equiv 2 \; \text{mod} \; 6$, then $b_r=a_r+1=c_r+1$.\\
    \item If $r \equiv 3 \; \text{mod} \; 6$, then $a_r=b_r-1=c_r-1$.\\
    \item If $r \equiv 4 \; \text{mod} \; 6$, then $c_r=a_r+1=b_r+1$.\\
    \item If $r \equiv 5 \; \text{mod} \; 6$, then $b_r=a_r-1=c_r-1$.
    
\end{itemize}
 \end{lemma}
 \begin{proof}
    We prove this via induction on $r$:
    
    \indent \textbf{Basis Case} $\mathbf{0 \leq r \leq 5}$: Below is a table of values for $a_r, b_r, c_r$ when $0 \leq r \leq 5$:
    \begin{center}
        \begin{tabular}{|c| c| c| c|} 
        \hline
        $r$ & $a_r$ & $b_r$ & $c_r$\\
         \hline  \hline
        0 & 1 & 0 & 0\\
         \hline
        1 & 1 & 1 & 0\\
         \hline
        2 & 1 & 2 & 1\\
         \hline
        3 & 2 & 3 & 3\\
         \hline
        4 & 5 & 5 &6\\
         \hline
        5 & 11 & 10 &11\\
         \hline
        \end{tabular}
    \end{center}
    
     \indent \textbf{Inductive Step:}  
Assume $r \equiv 0 \; \text{mod} \; 6$, the lemma is true for all $r'<r$, and let $z \in \mathbb{Z}^+$ such that $r=6z$. Then
\[a_{6z}=a_{6z-1}+c_{6z-1}.\]
By induction, this is
\[b_{6z-1}+1+c_{6z-1}\]
or $c_{6z}+1$.

\indent We also have by induction that
\[a_{6z}=a_{6z-1}+b_{6z-1}+1,\]
which is $b_{6z-1}+1$.

\indent We now prove $c_{1+6z}=a_{1+6z}-1=b_{1+6z}-1$. Note
\[c_{1+6z}=c_{6z}+b_{6z},\]
which, by the previous case, is
\[c_{6z}+a_{6z}-1,\]
or $a_{1+6z}-1$.

\indent The previous case also yields 
\[c_{1+6z}=a_{6z}-1+b_{6z}\]
which is $b_{6z+1}-1$.

\indent We can repeat this pattern for $y+6z$ where $ 2 \leq y \leq 5$ and the lemma will follow.  
\end{proof}

Using this, we can also prove a lemma that gives some insight into $P_m(3)$ that we will use later:
\begin{lemma}\label{6_divides_Period}
For $m \geq 3$, $6|P_m(3)$.
\end{lemma}
\begin{proof}
Assume $m \geq 3$. Take $\alpha \equiv 0 \; \text{mod} \; 6$ large enough so that 
\[D^{\alpha}(0,0,1)=(c_{\alpha}, b_{\alpha}, a_{\alpha}) \in K(\mathbb{Z}_m^3).\] Let $d=P_m(3)$. Then we should have that 
\[a_{\alpha+d} \equiv a_{\alpha} \; \text{mod} \; m\]
\[b_{\alpha+d} \equiv b_{\alpha} \; \text{mod} \; m\]
\[c_{\alpha+d} \equiv c_{\alpha} \; \text{mod} \; m.\]
Notice that $\alpha \equiv 0 \; \text{mod} \; 6$ gives us that $b_{\alpha}=c_{\alpha}=a_{\alpha}-1$. If $d \not\equiv 0 \; \text{or} \; 3 \; \text{mod} \; 6$, then $c_{\alpha+d} \neq b_{\alpha+d}$, but $b_{\alpha+d} \equiv b_{\alpha} \; \text{mod} \; m$, which implies
 $c_{\alpha+d} \not \equiv c_{\alpha} \; \text{mod} \; m$, which contradicts $d=P_m(3)$.
 
\indent If we have that $d \equiv 3 \; \text{mod} \; 6$, and $m>2$ then $a_{\alpha+d}+1=b_{\alpha+d}$,
Because 
\[b_{\alpha+d}\equiv b_{\alpha} \; \text{mod} \; m,\] 
and $b_{\alpha}=a_{\alpha}-1$,
this implies $a_{\alpha+d} \not \equiv a_{\alpha} \; \text{mod} \; m$, which again contradicts $d=P_m(3)$. Therefore $6|P_m(3)$ for $m>2$.
\end{proof}

\section{Possible Periods for $n=3$}\label{Secn=3}
\indent We can now prove Lemma \ref{TuplesLength6}.

\begin{proof}[Proof of Lemma \ref{TuplesLength6}]
Assume $\mathbf{u}=(x_1, x_2, x_3) \in \mathbb{Z}_m^3$ satisfies the sum condition.

\textbf{(1):} It suffices to show if $D(\mathbf{u})=(x_1', x_2', x_3')$ then $x_1'+x_2'+x_3' \equiv 0 \; \text{mod} \; m$. Notice $x_i'=x_i+x_{i+1}$ for $i=1,2$ and $x_3'=x_3+x_1$. Therefore
\[x_1'+x_2'+x_3'=2(x_1+x_2+x_3), \]
which is equivalent to $0 \; \text{mod} \; m$.

\textbf{(2):} First, 
\[D^2(\mathbf{u})=(x_1+2x_2+x_3,x_2+2x_3+x_1, x_3+2x_1+x_2).\]
 Notice that 
\[x_i+2x_{i+1}+x_{i+2}=x_{i+1}+(x_1+x_2+x_3), \]
which is equivalent to $x_{i+1} \; \text{mod} \; m$ for $1 \leq i \leq 3$, where the subscripts are reduced modulo $3$. Therefore
\[D^2(\mathbf{u})=(x_2, x_3, x_1)\]
or $H(\mathbf{u})$.

\textbf{(3):} Here, $D^6(\mathbf{u})=H^3(\mathbf{u})$, which is $\mathbf{u}$. Therefore $\mathbf{u} \in K(\mathbb{Z}_m^3)$.

\textbf{(4)-(5):} From the proof of (3), $\text{Per}(\mathbf{u})|6$. Therefore, $\text{Per}(\mathbf{u})=1, 2, 3 \; \text{or} \; 6$. Assume $(x_1, x_2, x_3)$ does not satisfy the uniformity condition and is not $\vec{\mathbf{0}}$.

\indent If $\text{Per}(\mathbf{u})=1,2$, then $\mathbf{u}=D^2(\mathbf{u})=H(\mathbf{u})$ implies $x_1=x_2=x_3$, which is a contradiction.

\indent If $\text{Per}(\mathbf{u})=3$, then $\mathbf{u}=D^3(\mathbf{u})=D(x_2, x_3, x_1)$ implies
\[x_2+x_3 \equiv x_1 \; \text{mod} \; m\]
\[x_1+x_3 \equiv x_2 \; \text{mod} \; m\]
\[x_1+x_2 \equiv x_3 \; \text{mod} \; m.\]
Plugging the third equation into the first yields
\[2x_2+x_1 \equiv x_1 \; \text{mod} \; m,\]
or that
\[2x_2 \equiv 0 \; \text{mod} \; m.\]
Similarly, $2x_3 \equiv 0 \; \text{mod} \; m$ and $2x_1 \equiv 0 \; \text{mod} \; m$.

\indent If $m$ is odd, then this only happens when $\mathbf{u}=\vec{\mathbf{0}}$, which contradicts our assumptions.

\indent If $m$ is even, this can only happen when  $x_1, x_2, x_3 \in \{0, \frac{m}{2}\}$, in which case, $\mathbf{u}$ is in the  transition graph given in Figure \ref{SmallTransgraph}.
\begin{figure}
\begin{tikzpicture}[node distance={20mm}, thick, main/.style = {draw, circle}]
\node[main] (1) {$(0,\frac{m}{2},\frac{m}{2})$};
\node[main] (2) [below right of=1] {$(\frac{m}{2},0,\frac{m}{2})$};
\node[main] (3) [below left of=1] {$(\frac{m}{2},\frac{m}{2},0)$};
\node[main] (4) [above of=1] {$(0,0,\frac{m}{2})$};
\node[main] (5) [below right of=2] {$(\frac{m}{2},0,0)$};
\node[main] (6) [below left of=3] {$(0,\frac{m}{2},0)$};

\draw[->] (1)--(2);
\draw[->] (2)--(3);
\draw[->] (3)--(1);
\draw[->] (4)--(1);
\draw[->] (5)--(2);
\draw[->] (6)--(3);
\end{tikzpicture}
\caption{Part of Transition Graph for $\mathbb{Z}_m^3$, $m$ Even}\label{SmallTransgraph}
\end{figure}
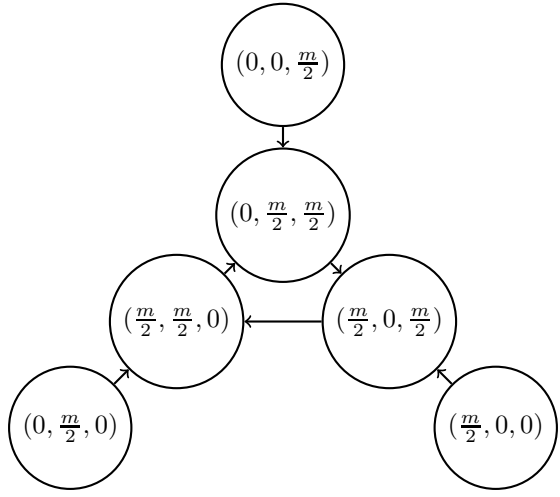

Note that there might be other tuples in $\mathbb{Z}_m^3$ whose Ducci sequences lead into this transition graph. Regardless, it is clear that if $x_1, x_2, x_3 \in \{0, \frac{m}{2}\}$, then $\text{Per}(\mathbf{u})=3$.

\end{proof}

This also yields the following corollary:
\begin{corollary}
For $m>2$ and $n=3$, a Ducci sequence with cycle of length $6$ always exists.
\end{corollary}

We now prove Theorem \ref{BigThm_prime}:
 \begin{proof}[Proof of Theorem \ref{BigThm_prime}]
 Let $n=3$ and $m$ be prime. We want to show that if $\mathbf{u}$ does not satisfy the sum or uniformity condition, then $\text{Per}(\mathbf{u})=P_m(3)$. 
 
\indent We first address the case where $m=3$. Here, $P_3(3)=6$. If you plot out the Ducci sequences for all $27$ tuples in $\mathbb{Z}_3^3$, you would find that the only times when $\text{Per}(\mathbf{u}) \neq 6$ for $\mathbf{u}=(x_1, x_2, x_3) \in \mathbb{Z}_3^3$, is when $x_1=x_2=x_3$. Here, tuples meeting the uniformity condition have period $2$.
 
\indent Assume that $m>3$ is prime. Recall that $6 |P_m(3)$ by Theorem \ref{6_divides_Period}. Then $a_{P_m(3)} \equiv 1 \; \text{mod} \; m$ and $b_{P_m(3)}=c_{P_m(3)} \equiv 0 \; \text{mod} \; m$. The first time that $a_r, b_r, c_r$ satisfy this property is when $r=P_m(3)$, as if there was a smaller number where this happened $s<P_m(3)$, then 
\[D^s(0,0,1)=(c_s, b_s, a_s)=(0,0,1)\]
 and $P_m(3)=s$ would follow.
 
\indent Suppose that there exists $\mathbf{u}=(x_1, x_2, x_3) \in \mathbb{Z}_m^3$ that does not satisfy the sum condition or the uniformity condition and $\text{Per}(\mathbf{u})=d<P_m(3)$. Then, because $L_m(3)=0$ by Theorem 7 of \cite{Paper3},
\begin{equation}\label{Eq_orig_eq1}
a_dx_1+b_dx_2+c_dx_3 \equiv x_1 \; \text{mod} \; m
\end{equation}
\begin{equation}\label{Eq_orig_eq2}
c_dx_1+a_dx_2+b_dx_3 \equiv x_2 \; \text{mod} \; m
\end{equation}
\begin{equation}\label{Eq_orig_eq3}
b_dx_1+c_dx_2+a_dx_3 \equiv x_3 \; \text{mod} \; m.
\end{equation}
Adding these together gives us 
\[(a_d+b_d+c_d)(x_1+x_2+x_3) \equiv (x_1+x_2+x_3) \; \text{mod} \; m.\]
Since $m$ is prime and $x_1+x_2+x_3 \not \equiv 0 \; \text{mod} \; m$, we see 
\begin{equation}\label{Eq_sum=1}
a_d+b_d+c_d \equiv 1 \; \text{mod} \; m.
\end{equation}
 If $d \equiv 0 \; \text{mod} \; 6$
    then 
    \[a_d+b_d+c_d=3b_d+1\]
    by Lemma \ref{n=3_coeff_relat}.
    By Equation (\ref{Eq_sum=1}),
    \[3b_d+1 \equiv 1 \; \text{mod} \; m\]
    and
    \[3b_d \equiv 0 \; \text{mod} \; m.\]
     Because $m>3$ is prime, this yields
    \[b_d \equiv 0 \; \text{mod} \; m.\]
   Since $d \equiv 0 \; \text{mod} \; 6$, $c_d \equiv 0 \; \text{mod} \; m$ and $a_d \equiv 1 \; \text{mod} \; m$. This contradicts $P_m(3)>d$. Therefore, $\text{Per}(\mathbf{u})=P_m(3)$.
   
\indent Suppose $d \not \equiv 0 \; \text{mod} \; 6$. We have that $d|P_m(3)$, so take $d_1=2d, 3d, \;\text{or} \; 6d$ so that $d_1 \equiv 0 \; \text{mod} \; 6$ and $d_1$ is as small as possible. If $d_1<P_m(3)$, then we repeat the above argument for $d_1$ and conclude that $\text{Per}(\mathbf{u})=P_m(3)$. If you cannot take $d_1$ such that $d_1<P_m(3)$, then this means that $P_m(3) \in \{2d,3d,6d\}$. We look at these as three different cases.

\textbf{Case 1} $\mathbf{P_m(3)=2d}$: In this case, $d \equiv 3 \; \text{mod} \; 6$. Then, by Lemma \ref{n=3_coeff_relat}, $b_d=c_d$ and $a_d=b_d-1$. Note that $b_{P_m(3)}\equiv 0 \; \text{mod} \; m$ because 
\[D^{P_m(3)}(0,0,1)=(0,0,1).\]
Therefore, 
\[b_{2d} \equiv 0 \; \text{mod} \; m.\]
However, if we break down $b_{2d}$ using Corollary \ref{coefficient_sum_n=3}, we have
\[b_{2d}=2a_db_d+c_d^2.\]
Substituting $a_d,c_d$ in terms of $b_d$, this is
\[2(b_d-1)b_d+b_d^2,\]\
or
\[b_d(3b_d-2), \]
which we know is equivalent to $0 \; \text{mod} \; m$.
So either $b_d \equiv 0 \; \text{mod} \; m$ or $3b_d-2 \equiv 0 \; \text{mod} \; m$. If $b_d \equiv 0 \; \text{mod} \; m$, then $c_d \equiv 0 \; \text{mod} \; m$ and $a_d \equiv -1 \; \text{mod} \; m$, which contradicts Equation (\ref{Eq_sum=1}). 

\indent Assume $3b_d-2 \equiv 0 \; \text{mod} \; m$. Then, 
 \[3b_d \equiv 2 \; \text{mod} \; m\]
 \[3c_d \equiv 2 \; \text{mod} \; m\]
 \[3a_d \equiv -1 \; \text{mod} \; m.\]
 
 Multiplying Equations (\ref{Eq_orig_eq1}), (\ref{Eq_orig_eq2}), and (\ref{Eq_orig_eq3}) by $3$ and substituting in these equivalencies, this yields
 \[-x_1+2x_2+2x_3 \equiv 3x_1 \; \text{mod} \; m,\]
 \[2x_1-x_2+2x_3 \equiv 3x_2 \; \text{mod} \; m\]
 \[2x_1+2x_2-x_3 \equiv 3x_3 \; \text{mod} \; m.\]
 Moving all of the $x_i$ to one side and dividing by $2$ gives us
 \[-2x_1+x_2+x_3 \equiv 0 \; \text{mod} \; m\]
 \[x_1-2x_2+x_3 \equiv 0 \; \text{mod} \; m\]
 \[x_1+x_2-2x_3 \equiv 0 \; \text{mod} \; m.\]
 Next, solve for $x_1$ in the second two equations and set them equal to produce
  \[-x_2+2x_3 \equiv 2x_2-x_3 \; \text{mod} \; m\]
 or
 \[x_2 \equiv x_3 \; \text{mod} \; m.\]
 Finally, plugging this into the first equation leads us to 
 \[x_1 \equiv x_2 \; \text{mod} \; m \equiv x_3 \; \text{mod} \; m.\]
  However, we were assuming that $\mathbf{u}$ did not meet the uniformity condition, and we have a contradiction. Hence, $P_m(3)\neq 2d$.

\textbf{Case 2} $\mathbf{P_m(3)=3d}$: In this case, we either have $d \equiv 2 \; \text{mod} \; 6$ or $d \equiv 4 \; \text{mod} \; 6$, which we handle separately.

\begin{itemize}
    \item If $d \equiv 2 \; \text{mod} \; 6$, then $a_d=c_d$ and $b_d=c_d+1$. 
    Therefore, Equation (\ref{Eq_sum=1}) gives us $3c_d+1 \equiv 1 \; \text{mod} \; m$ and therefore $c_d \equiv 0 \; \text{mod} \; m$.
    
    \indent This produces $a_d\equiv 0 \; \text{mod} \; m$ and $b_d \equiv 1 \; \text{mod} \; m$ which means that
    \[x_2 \equiv x_1 \; \text{mod} \; m\]
    \[x_3 \equiv x_2 \; \text{mod} \; m\]
    \[x_1 \equiv x_3 \; \text{mod} \; m.\]
    But we were assuming that $x_1, x_2, x_3$ not all equal, so we have a contradiction.\\
    \item If $d \equiv 4 \; \text{mod} \; 6$, then $a_d=b_d$, $c_d=b_d+1$. 
    Using Equation (\ref{Eq_sum=1}) yields $3b_d+1 \equiv 1 \; \text{mod} \; m$ and $b_d \equiv 0 \; \text{mod} \; m$. Therefore, 
    \[x_3 \equiv x_1 \; \text{mod} \; m\]
    \[x_1 \equiv x_2 \; \text{mod} \; m\]
    \[x_2 \equiv x_3 \; \text{mod} \; m.\]
    But we were assuming $x_1, x_2, x_3$ not all equal. Therefore we have a contradiction.
\end{itemize}
In either case, we cannot have $P_m(3)=3d$.

\textbf{Case 3} $\mathbf{P_m(3)=6d}$: In this case, either $d \equiv 1 \; \text{mod} \; 6$ or $d \equiv 5 \; \text{mod} \; 6$. Again, we handle these separately.
 \begin{itemize}
 \item Assume $d \equiv 1 \; \text{mod} \; 6$. Then we have $a_d=b_d$ and $c_d=b_d-1$. This gives us 
 \[a_d+b_d+c_d=3b_d-1.\]
 From Equation (\ref{Eq_sum=1}), this produces
 \[3b_d-1 \equiv 1 \; \text{mod} \; m,\]
and
 \[3b_d \equiv 2 \; \text{mod} \; m\]
 \[3a_d \equiv 2 \; \text{mod} \; m\]
 \[3c_d \equiv -1 \; \text{mod} \; m.\]
 Multiply Equations (\ref{Eq_orig_eq1}), (\ref{Eq_orig_eq2}), and (\ref{Eq_orig_eq3}) by $3$ and substitute these equivalencies to produce
 \[2x_1+2x_2-x_3 \equiv 3x_1 \; \text{mod} \; m\]
 \[-x_1+2x_2+2x_3 \equiv 3x_2 \; \text{mod} \; m\]
 \[2x_1-x_2+2x_3 \equiv 3x_3 \; \text{mod} \; m.\]
 Rearranging our terms yields the equations
 \begin{equation}\label{Eq_new1}
 -x_1+2x_2-x_3 \equiv 0 \; \text{mod} \; m
 \end{equation}
 \begin{equation}\label{Eq_new2}
 -x_1-x_2+2x_3 \equiv 0 \; \text{mod} \; m
 \end{equation}
 \begin{equation}\label{Eq_new3}
 2x_1-x_2-x_3 \equiv 0 \; \text{mod} \; m.
 \end{equation}
If we use Equation (\ref{Eq_new3}) to solve for $x_3$ and plug this into Equation (\ref{Eq_new1}), we have
\[-3x_1+3x_2 \equiv 0 \; \text{mod} \; m,\]
which forces $x_1 \equiv x_2 \; \text{mod} \; m$. Plugging this back into Equation (\ref{Eq_new3}) leads to $x_1 \equiv x_3 \; \text{mod} \; m$, which means $x_1 \equiv x_2 \; \text{mod} \; m \equiv x_3 \; \text{mod} \; m$ and we have a contradiction.

\item Assume $d \equiv 5 \; \text{mod} \; 6$. Then $a_d=c_d$ and $b_d=c_d-1$. Therefore, 
\[a_d+b_d+c_d=3c_d-1.\]
Because of Equation (\ref{Eq_sum=1}), we have 
\[3c_d-1 \equiv 1 \;\text{mod} \; m.\]
This produces 
\[3c_d \equiv 2 \; \text{mod} \; m\]
\[3a_d \equiv 2 \; \text{mod} \; m\]
\[3b_d \equiv -1 \; \text{mod} \; m.\]
We once more multiply Equations (\ref{Eq_orig_eq1}), (\ref{Eq_orig_eq2}), and (\ref{Eq_orig_eq3}) by $3$ and use these equivalencies to yield
\[2x_1-x_2+2x_3 \equiv 3x_1 \; \text{mod} \; m\]
\[2x_1+2x_2-x_3 \equiv 3x_2 \; \text{mod} \; m\]
\[-x_1+2x_2+2x_3 \equiv 3x_3 \; \text{mod} \; m.\]
However, these are Equations (\ref{Eq_new1}), (\ref{Eq_new2}), and (\ref{Eq_new3}) after moving all of the $x_i$ to the left side, which still results in $x_1 \equiv x_2 \; \text{mod} \; m \equiv x_3 \; \text{mod} \; m$ and a contradiction.
 
 \end{itemize}
 Therefore, the only possibility is that $\text{Per}(\mathbf{u})=P_m(n)$ and the theorem follows.
 \end{proof}

\section{Examining Possible Periods when $n,m$ Prime}\label{sec_not_n=3}
\indent For the remainder of the paper, we will consider when $n,m$ are prime, but allow for $n \neq 3$.

\indent As discussed in Sections \ref{Background} and \ref{sec_not_n=3}, there are cases where tuples satisfying the sum condition generate a smaller period. This is certainly true for when $n=3$ and all $m$ when $P_m(3)>6$ by Lemma \ref{TuplesLength6}. However, this is not always the case. 

\indent To further investigate this, we present a few cases where $n,m$ prime. Figure \ref{Table1} provides cases where $n>3$, $n,m$ are both prime, and all tuples that are not $\vec{\mathbf{0}}$ and do not satisfy the uniformity condition generate the maximum period. The value of $P_m(n)$ and the period of all tuples that do satisfy the uniformity condition has also been included.

\begin{figure}
\centering

\begin{tabular}{|c|c|c|c|}
\hline
$n$ & $m$ & $P_m(n)$ & $x_i \neq 0$ all equal\\
\hline
5 & 3 & 40 & 2\\
  & 5 & 20 & 4\\
\hline
7 & 7 & 21 & 3\\
\hline
11 & 3 & 242 & 2\\
   & 11 & 110 & 10\\
\hline
13 & 3 & 26 & 2\\
   & 13 & 156 & 12\\
   & 17 & 63856 & 8\\
   & 29 & 24388 & 28\\
 \hline
 17 & 3 & 27880 & 2\\
    & 17 & 136 & 8\\
  \hline
  19 & 3 & 373958 & 2\\
     & 19 & 342 & 18\\
  \hline

\end{tabular}
\caption{$n,m$ prime, three possible periods}\label{Table1}
\end{figure}

\indent Figure \ref{Table2} provides cases where $n>3$ and $n,m$ prime and the only tuples with periods smaller than the maximum are $\vec{\mathbf{0}}$ and those that satisfy the uniformity or sum condition. We are including the periods that the tuples satisfying the sum condition generate, in addition to the information provided in Figure \ref{Table1}.

\indent For the cases in Figure \ref{Table2}, note that the periods of the tuples satisfying the sum condition vary, even for a fixed $n$, unlike what we see when $n=3$.

\begin{figure}
\centering
\begin{tabular}{|c|c|c|c|c|}
\hline
n & m & $P_m(n)$ & $x_i \neq 0$ all equal & $\displaystyle{\sum_{i=1}^n x_i \equiv 0 \; \text{mod} \; m}$\\
\hline 
5 & 7 & 240 & 3 & 80\\
  & 13 & 420 & 12 & 140\\
  & 17 & 360 & 8 & 180\\
  & 23 & 2640 & 11 & 240\\
  & 29 & 140 & 28 & 70\\
  \hline
7 & 3 & 182 & 2 & 91\\
  & 5 & 868 & 4 & 217\\
  & 17 & 17192 & 8 & 2149\\
  & 19 & 16002 & 18 & 889\\
  & 23 & 6083 & 11 & 553\\
\hline
11 & 5 & 3124 & 4 & 1562\\
   & 7 & 184866 & 3 & 16806\\
   & 13 & 4084212 & 12 & 680702\\
   & 17 & 7809208 & 8 & 1952302\\
   & 19 & 27237078 & 18 & 3026342\\
\hline
13 & 7 & 509808 & 3 & 169936\\
   & 11 & 7676760 & 10 & 1535352\\
   \hline

\end{tabular}
\caption{$n,m$ Prime, four possible periods}\label{Table2}
\end{figure}

\indent We now discuss how we made these conclusions. First, we run a program in MATLAB (from \cite{MATLAB}) that can determine the value of $P_m(n)$ for a given $n$ and $m$.
Observing Figures \ref{Table1} and \ref{Table2}, the value of $P_m(n)$ typically increases as $n,m$ increase. As a result, for many $n,m$, $P_m(n)$ becomes too large for MATLAB to find, preventing us from testing these cases. Our current program is currently set to stop trying to calculate $P_m(n)$ after the program determines it is larger than $15$ million.

\indent Once we find $P_m(n)$, Lemma 1 of \cite{Breuer1} tells us that $\text{Per}(\mathbf{u})|P_m(n)$ for every $\mathbf{u} \in \mathbb{Z}_m^n$. We can then create a list of possible periods for $\mathbf{u}$ out of the divisors of $P_m(n)$.

\indent We now consider the example of $n=5$ and $m=7$ to demonstrate our strategy. Because $P_7(5)=240$, the possible periods are 
\[\{1,2,3,4,5,6,8,10,12,15, 16,20,24, 30, 40, 48, 60, 80, 120, 240\}.\]
 We also take advantage of Theorem 7 of \cite{Paper3}, which says $L_m(n)=0$. This is notable because for any tuple $\mathbf{u} \in \mathbb{Z}_7^5$, we have
\[D^{\text{Per}(\mathbf{u})}(\mathbf{u})=\mathbf{u}.\]
If we wish to test whether a divisor of $P_m(n)$, $d$, satisfies $\text{Per}(\mathbf{u})=d$ for some $\mathbf{u} \in \mathbb{Z}_7^5$, we want to determine if there exists $(x_1, x_2, x_3, x_4, x_5)$ such that 
\[D^d(x_1, x_2, x_3, x_4, x_5)=(x_1, x_2, x_3, x_4, x_5).\]
We will start by testing $40$. Notice the values of $a_{40,s}$ are as follows:
\begin{center}
\begin{tabular}{|c|c|c|c|c|c|}
\hline
$s$ & 1 & 2 & 3 & 4 & 5\\
\hline
$a_{40,s}$ & 1 & 2 & 2 & 2 & 2\\
\hline
\end{tabular}
\end{center}

If $\text{Per}(x_1, x_2, x_3, x_4, x_5)=40$, then $(x_1, x_2, x_3, x_4, x_5)$ satisfies the system of equations
\[a_{40,1}x_i+a_{40,2}x_{i+1}+ \cdots +a_{40,5}x_{i-1} \equiv x_i \text{mod} \; 7\]
for $1 \leq i \leq 5$, or
\[A_1
\begin{bmatrix}
x_1\\
x_2\\
x_3\\
x_4\\
x_5

\end{bmatrix}
=\begin{bmatrix}
0\\
0\\
0\\
0\\
0

\end{bmatrix}
\]
where 
\[A_1=\begin{bmatrix}
0 & 2 & 2 & 2 & 2 \\
2 & 0 & 2 & 2 & 2\\
2 & 2 & 0 & 2 & 2 \\
2 & 2 & 2 & 0 & 2 \\
2 & 2 & 2 & 2 & 0 
\end{bmatrix}.\]
\indent Since $det(A_1)=128$, $A_1$ is invertible, so there is only one solution. Because $D^{40}(\vec{\mathbf{0}})=\vec{\mathbf{0}}$,
this one solution must be $\vec{\mathbf{0}}$. Since $\text{Per}(\vec{\mathbf{0}})=1$, we can determine that there is not a tuple $(x_1, x_2, x_3, x_4, x_5)$ such that $\text{Per}(x_1, x_2, x_3, x_4, x_5)=40$

\indent However, we must also bear in mind that if there exists $\mathbf{u} \in \mathbb{Z}_m^n$ such that $\text{Per}(\mathbf{u})=d$ where $d|d^*$, and $d^*|240$, then  $D^{d^*}(\mathbf{u})=\mathbf{u}$.
Therefore, this also confirmed that $2,4,5,8,10$ are also not periods for any tuple in $\mathbb{Z}_7^5$ because they are divisors of $40$.

\indent Now, let us test $120$ and its divisors as possible periods. Here, we aim to solve the system of equations

\[A_2
\begin{bmatrix}
x_1\\
x_2\\
x_3\\
x_4\\
x_5
\end{bmatrix}
=\begin{bmatrix}
0\\
0\\
0\\
0\\
0
\end{bmatrix}\]

where 
\[A_2=\begin{bmatrix}
4 & 6 & 6 & 6 & 6 \\
6 & 4 & 6 & 6 & 6 \\
6 & 6 & 4 & 6 & 6 \\
6 & 6 & 6 & 4 & 6 \\
6 & 6 & 6 & 6 & 4 

\end{bmatrix}.\]
Because $det(A_2)=448$, $A_2$ is not invertible so there are more solutions than $\vec{\mathbf{0}}$. However, recall that if $\mathbf{u}$ satisfies the uniformity condition, then $\text{Per}(\mathbf{u})=3$. As a result, if $3|d$, then $D^d(\mathbf{u})=\mathbf{u}$. Since $3|120$, all tuples satisfying the uniformity condition are a solution to this system of equations. We now want to confirm that the only solutions to this system of equations are $\vec{\mathbf{0}}$ and tuples satisfying the uniformity condition.

 \indent To begin, we make the assumption that $x_1=1$ and rewrite our system of equations, with the goal now being to solve 
\[\begin{bmatrix}
6 & 6 & 6 & 6 \\
4 & 6 & 6 & 6 \\
6 & 4 & 6 & 6 \\
6 & 6 & 4 & 6 

\end{bmatrix}
\begin{bmatrix}
x_2\\
x_3\\
x_4\\
x_5
\end{bmatrix}
=\begin{bmatrix}
-4\\
-6\\
-6\\
-6
\end{bmatrix}.
\]
Since $det(B)=48$, $B$ is invertible. So there is only one solution to the system of equations when $x_1=1$, which is $(1,1,...,1)$.
We can repeat this process when $x_1=j$ for $2 \leq j \leq 6$ and find that $x_i=j$ for all $1 \leq i \leq 5$. In other words, only $\vec{\mathbf{0}}$ and those tuples satisfying the uniformity condition, which we already know have period $3$, are solutions to this system of equations.
Therefore, $120$ is not a period for any tuple, as well as $2, 4, 5, 6, 8, 10, 12, 15, 20, 24, 30, 40, 60$.

\indent Using these methods, we can eliminate all of the periods except for $80$. The $a_{r,s}$ coefficients in this case are

\begin{center}
\begin{tabular}{|c|c|c|c|c|c|}
\hline
$s$ & 1 & 2 & 3 & 4 & 5 \\
\hline
$a_{80,s}$ & 3 & 2 & 2 & 2 & 2 \\
\hline
\end{tabular}
\end{center}

When setting up the system of equations, each equation becomes 
\[2(x_1+ x_2+ x_3+ x_4+ x_5) \equiv 0 \; \text{mod} \; 7.\]
Therefore, if $x_1+x_2+ x_3+x_4+x_5 \equiv 0 \; \text{mod} \; 7$, then 
\[D^{80}(x_1, x_2, x_3, x_4, x_5)=(x_1, x_2, x_3, x_4, x_5).\]
Since we eliminated the other divisors of $80$ as periods, $\text{Per}(x_1, x_2, x_3, x_4, x_5)=80$. Therefore, all tuples satisfying the sum condition have period $80$.  

\indent These methods were used to create Figures \ref{Table1} and \ref{Table2}. 

\indent  Finally, we provide a few examples of $n,m$ prime that fall into neither of the cases presented in Figures \ref{Table1} or \ref{Table2}, which we determined using the above methods. Namely, we found divisors of $P_m(n)$ that had solutions to their system of equations that were not $\vec{\mathbf{0}}$, did not satisfy the uniformity condition, and did not satisfy the sum condition.

\indent Figure \ref{Table3} provides these examples, as well as the periods that came up as exceptions.

\begin{figure}
\centering
\begin{tabular}{|c|c|c|c|c|}
\hline
$n$ & $m$ & $P_m(n)$ & $x_i \neq 0$ all equal & Exceptions\\
\hline
5 & 11 & 10 & 10 & 2,5\\
  & 19 & 90 & 18 & 45\\
  & 31 & 30 & 5 & 3, 15\\
 \hline
7 & 13 & 84 & 12 & 28\\
\hline
11 & 23 & 22 & 11 & 11\\
\hline 
13 & 5 & 312 & 4 & 156\\
   & 23 & 158158 & 11 & 11297\\
\hline
23 & 3 & 177146 & 2 & 88573\\
\hline
\end{tabular}
\caption{$n,m$ prime, Exceptions}\label{Table3}
\end{figure}

\indent Out of these exceptions, we are able to look more in depth to the cases when $n=5$ and $m=11,19$ and when $n=7$ and $m=13$, as $\mathbb{Z}_m^n$ is small enough for MATLAB to  provide a file with all of the tuples in $\mathbb{Z}_m^n$ and their periods in a reasonable amount of time. For all of these cases, all tuples that generate a smaller period satisfy the sum condition. However, there were tuples that satisfied the sum condition that generated the maximum period.

\indent We now look more into the case where $n=5$, $m=11$. Suppose that you take a tuple in $\mathbb{Z}_{11}^5$ with period $5$ and you allow $S_5$ to act on that tuple such that for $\phi \in S_5$, 
\[\phi \circ (x_1, x_2, ..., x_5)=(x_{\phi(1)}, x_{\phi(2)}, ..., x_{\phi(5)}).\]
Let $J$ be the set of all $\phi \in S_5$ such that if $\text{Per}(x_1, x_2, ..., x_5)=5$, then 
\[\text{Per}(x_{\phi(1)}, x_{\phi(2)}, ..., x_{\phi(5)})=5.\]
By using MATLAB to find $J$, we determined that $J \cong D_{10}$. This can similarly be shown for when $n=5$, $m=19$ and $n=7$, $m=13$ and the period that appeared as exceptions for each case: Let $d$ be the smaller exception period for $\mathbb{Z}_m^n$ where either $n=5$ and $m=19$, or $n=7$ and $m=13$. Let $J$ to be the set of all $\phi \in S_n$ such that if $\text{Per}(x_1, x_2, ..., x_n)=d$, then $\text{Per}(x_{\phi(1)}, x_{\phi(2)}, ..., x_{\phi(5)})=d$. Then $J \cong D_{2n}$.

\indent This is not the case for the tuples in $\mathbf{u} \in \mathbb{Z}_{11}^5$ such that $\text{Per}(\mathbf{u})=2$. All of these tuples satisfy
\[\mathbf{u}=z*(1,9,4,3,5)\]
where $z \in \mathbb{Z}_m$, $z \neq 0$.

\indent For the cases where $n,m$ are prime that we have investigated, we found one case where $n,m$ are prime and every tuple that satisfies the sum condition generates a smaller period than the maximum and that also has other possible periods. This is when $n=7$ and $m=11$. Here, $P_7(11)=1330$ and tuples that satisfy the uniformity condition have period $10$. There are $1330$ tuples with period $19$, all of which satisfy the sum condition. The rest of the tuples satisfying the sum condition have period $133$. There are $13300$ tuples with period $190$, none of which satisfy the sum condition. Let $J_1$ be the set of all $\phi \in S_7$ such that if $\text{Per}(x_1, x_2, ..., x_7)=19$ then $\text{Per}(x_{\phi(1)}, x_{\phi(2)}, ..., x_{\phi(7)})=19$. Let $J_2$ be the set of all $\phi \in S_7$ such that if $\text{Per}(x_1, x_2, ..., x_7)=190$ then $\text{Per}(x_{\phi(1)}, x_{\phi(2)}, ..., x_{\phi(7)})=190$. Then we find that $J_1=J_2 $ and they are isomorphic to the Frobenius group of order 21.

\section{Periods when $n=m$ is Prime}
We draw our attention back to Figure \ref{Table1} to note that all of the cases when $n=m$ for $n$ prime, $5 \leq n \leq 19$ fall into the case where all tuples except $\vec{\mathbf{0}}$ and those that satisfy the uniformity condition generate the maximum period. This leads us to Theorem \ref{BigThm_n_not_3}. Before we prove it, we prove a lemma.

\begin{lemma}\label{j_by_j_matrix}
    Consider an $j \times j $ matrix that follows this pattern:
    \[B=
    \begin{bmatrix}
    -1& 1& -1& 1 &\cdots & \pm 1\\
    0 & -1 & 1 &-1 & \cdots & \mp 1\\
    1 & 0 & -1 & 1 & \cdots & \pm 1\\
    \vdots & & & & \ddots & \vdots\\
    \mp 1 & \pm 1 & \mp 1 & \pm 1 & \cdots & -1
    \end{bmatrix}
    ,\]
    where $\pm$ is positive if $j$ is even and $\mp$ is positive if $j$ is odd. Then 
    \[det(B)= 
    \begin{cases}
       1 & j \; \text{is even}\\
       -1 & j \; \text{is odd}
    \end{cases}.\]
\end{lemma}
\begin{proof}
    We prove this via induction on $j$, with $j=2$ and $j=3$ as our basis cases.
    
   \indent \textbf{Inductive case:} Assume that the hypothesis is true for $j'<j$. Notice that the matrix made by removing the first row and first column is the matrix that follows our pattern for $j-1$. By induction, the determinant of this $(j-1) \times (j-1)$ matrix is $-1$ if $j$ is even and $1$ if $j$ is odd. 
    
\indent Also, if we remove the first column and $i$th row for $3 \leq i \leq j$, then the top two remaining rows are
    \[
    \begin{bmatrix}
        1& -1& 1 &\cdots & \pm 1\\
    -1 & 1 &-1 & \cdots & \mp 1
    \end{bmatrix}
    .\]
    Notice that the second row is the first row multiplied by $-1$. Therefore, the determinant of any matrix made by this method has determinant $0$.
    
  \indent If we put this altogether, we get
    \[det(B)=-1*M_{1,1}-0*M_{2,1}+1*M_{3,1} + \cdots +1*M_{j,1}\]
    \[=\begin{cases}
       1 & j \; \text{is even}\\
       -1 & j \; \text{is odd}
    \end{cases},\]
where $M_{i,j}$ is the determinant of the minor of row $i$, column $j$.   
 
   \indent From here, the lemma follows.
\end{proof}

\indent We can now prove Theorem \ref{whensmallperiod_n_not3}.
\begin{proof}[Proof of Theorem \ref{BigThm_n_not_3}]
Let $\delta$ be the multiplicative order of $2 \; \text{mod} \; m$. Let $n=m=p$ be prime.

    \indent $\mathbf{(1)}$: Only $\vec{\mathbf{0}}$ can satisfy $\text{Per}(\mathbf{u})=1$ for $n=m=p$ odd prime because $L_p(p)=0$ by Theorem 2 of \cite{Paper3}, meaning $\vec{\mathbf{0}}$ only has one predecessor, which is itself.
    
   \indent $\mathbf{(2)}$: By Lemma \ref{uniformcondlemma}, tuples satisfying the uniformity condition have period $\delta$. 
   
   \indent $\mathbf{(3)}$: First, we aim to find $P_p(p)$. Note that because of Corollary 6 in \cite{Paper1} $a_{p,1}=2$ and 
   \[a_{p,s} = \binom{p}{s-1},\] 
  so $a_{p,s} \equiv 0 \; \text{mod} \; p$ for $s \neq 1$. This means that $D^p(\mathbf{u})=2(\mathbf{u})$ for every $\mathbf{u} \in \mathbb{Z}_p^p$. Therefore, the smallest $\alpha$ such that $D^{\alpha}(0,0,...,0,1)=(0,0,...,0,1)$ is when $\alpha=p*\delta$ and $P_p(p)=p*\delta$.
   
   \indent We now finally aim to prove that if $\mathbf{u} \in \mathbb{Z}_p^p$ does not satisfy the uniformity condition and is not $\vec{\mathbf{0}}$, then $\text{Per}(\mathbf{u})=P_p(p)$. The possible period lengths for any tuple in $\mathbf{u}$ divide $P_p(p)$ and therefore are divisors of $\delta$ or divisors of $\delta$ times $p$ since $p$ is prime. Note that we cannot have a tuple with a period that is a multiple of $p$ and smaller than $\delta*p$, since this would mean that there exists $j<\delta$ such that $2^j\mathbf{u}= \mathbf{u}$ and therefore $2^j \equiv 1 \; \text{mod} \; p$, a contradiction. Therefore, we only need to consider period lengths that are divisors of $\delta$. If $\mathbf{u} \in \mathbb{Z}_p^p$ and $\text{Per}(\mathbf{u})=d$ for some $d |\delta$, then $D^{\delta}(\mathbf{u})=\mathbf{u}$. 
   Also, $\delta| \phi(p)=p-1$, so if $\text{Per}(\mathbf{u})|\delta$, then $D^{p-1}(\mathbf{u})=\mathbf{u}$.
   Therefore, we only need to prove that if $D^{p-1}(\mathbf{u})=\mathbf{u}$, then $\mathbf{u}=\vec{\mathbf{0}}$ or $\mathbf{u}$ satisfies the uniformity condition. 
   
   \indent Note
    \[a_{p-1,s} \equiv 
    \begin{cases}
        1 \; \text{mod} \; p & s \; \text{is odd}\\
        -1 \; \text{mod} \; p & s \; \text{is even}
    \end{cases}\]
    because $a_{p-1,s}=\displaystyle{\binom{p-1}{s-1}}$ and because of the well-known fact that 
    \[\binom{p-1}{s-1} \equiv (-1)^{s-1}\; \text{mod} \; p,\]
     a proof of which is in Lemma 4 of \cite{Paper4}.
    Assume for some $(x_1, x_2, ..., x_p) \in \mathbb{Z}_p^p$, $D^{p-1}(x_1, x_2, ..., x_p)=(x_1, x_2, ..., x_p)$. Then 
    \[x_1 -x_2 +x_3 - \cdots +x_p \equiv x_1 \; \text{mod} \; p\]
    \[x_2-x_3+x_4- \cdots +x_1 \equiv x_2 \; \text{mod} \; p\]
    \[\vdots\]
    \[x_p-x_1+ x_2- \cdots + x_{p-1} \; \text{mod} \; p.\]
    We can then use the following matrices to solve the system of equations:
    \[
    \begin{bmatrix}
        0 & -1 &1 & \cdots & -1 & 1\\
        1 & 0 & -1 & \cdots & 1 & -1\\
        \vdots & & & \ddots & & \vdots\\
        -1 & 1 & -1 & \cdots & 1 & 0
    \end{bmatrix}    
    \begin{bmatrix}
        x_1\\
        x_2\\
        \vdots\\
        x_p
    \end{bmatrix}    
    =
    \begin{bmatrix}
        0\\
        0\\
        \vdots\\
        0
    \end{bmatrix}
    .\]
    Since $(x,x,...,x)$ satisfies this system of equations for every $x \in \mathbb{Z}_p$, the matrix 
    \[
    \begin{bmatrix}
        0 & -1 &1 & \cdots & -1 & 1\\
        1 & 0 & -1 & \cdots & 1 & -1\\
        \vdots & & & \ddots & & \vdots\\
        -1 & 1 & -1 & \cdots & 1 & 0
    \end{bmatrix}
    \]
    must be singular. We will now fix $x_1=1$ and remove the last equation from the system to make a new system of equations with matrices
    \[
    \begin{bmatrix}
         -1 &1 & \cdots & -1 & 1\\
         0 & -1 & \cdots & 1 & -1\\
        \vdots & &  \ddots & & \vdots\\
        1 & -1  & \cdots & 0 & -1
    \end{bmatrix}   
    \begin{bmatrix}
        x_2\\
        x_3\\
        \vdots\\
        x_p
    \end{bmatrix}     
    =
    \begin{bmatrix}
        0\\
        -1\\
        \vdots\\
        1
    \end{bmatrix}.
    \]
   \indent We already know that $(1,1,...,1)$ is a solution that satisfies this system. If our new matrix 
    \[B=
    \begin{bmatrix}
         -1 &1 & \cdots & -1 & 1\\
         0 & -1 & \cdots & 1 & -1\\
        \vdots & & \ddots & & \vdots\\
        1 & -1  & \cdots & 0 & -1
    \end{bmatrix}
    \]
    is nonsingular, then that means $(1,1,...,1)$ is the only solution when $x_1=1$ is fixed. We can repeat these steps for $x_1=x$ for any $x \in \mathbb{Z}_p$ and reach the conclusion that the only solutions to our original system of equations are the tuples that satisfy the uniformity condition and $\vec{\mathbf{0}}$. Our theorem would then follow from this.
    
   \indent Notice that $B$ here follows the pattern of the matrix described in Lemma \ref{j_by_j_matrix}. Therefore, by this lemma,  $det(B)$ is nonzero in $\mathbb{Z}_p$, and therefore nonsingular. From here, our theorem follows. 
    
\end{proof}

\end{document}